\documentclass[a4paper,12pt]{article}
\usepackage{color}
\usepackage{amsmath,amssymb,amsthm,graphics,latexsym,amsfonts}
\usepackage{fancyhdr}
\usepackage{color}
\usepackage{graphics}
\usepackage{epsfig}
\usepackage{epstopdf}
\renewcommand{\figurename}

\title{\Large Maximum odd induced subgraph of a graph concerning its chromatic number
\thanks{The work was supported by NSFC (No. 12061073)}}
\author{ {Tao Wang, Baoyindureng Wu \footnote{Corresponding author.
Email: baoywu@163.com (B. Wu) }}\\
\small  College of Mathematics and System Sciences, Xinjiang
University \\ \small  Urumqi, Xinjiang 830046, P.R.China \\}
\date{}

\newtheorem{theorem}{Theorem}[section]
\newtheorem{lemma}[theorem]{Lemma}
\newtheorem{corollary}[theorem]{Corollary}

\newtheorem{conjecture}[theorem]{Conjecture}

\usepackage{indentfirst}

\begin{document}
\maketitle {\small \noindent{\bfseries Abstract}
Let $f_{o}(G)$ be the maximum order of an odd induced subgraph of $G$. In 1992, Scott proposed a conjecture that $f_{o}(G)\geq \frac {n} {\chi(G)}$ for a graph $G$ of order $n$ without isolated vertices, where $\chi(G)$ is the chromatic number of $G$. In this paper, we show that the conjecture is not true for bipartite graphs, but is true for all line graphs.
In addition, we also disprove a conjecture of Berman, Wang and Wargo in 1997, which states that $f_{o}(G)\geq 2\lfloor\frac {n} {4}\rfloor$ for a connected graph $G$ of order $n$. Scott's conjecture is open for graphs with chromatic number at least 3.


\section {\large Introduction}
All graphs considered here are simple and finite. For graph-theoretic notation not explained in this paper, we refer to \cite{Bondy}. Let $G=V(G), E(G))$ be a graph. The order and size of $G$ are often denoted by $v(G)$ and $e(G)$. We use $G-S$ to denote the graph obtained from $G$ by removing vertices of $S$ and all edges that are incident to a vertex of $S$. The subgraph $G-(V(G)\setminus S)$ is said to be the induced subgraph of $G$ induced by $S$, and
is denoted by $G[S]$. We use $e_G(A, B)$ to denote the number of edges with one end in $A$ and the other end in $B$ for $A, B\subseteq V(G)$. 
The distance of two vertices $u$ and $v$, denoted by $d_G(u,v)$, is the length of shortest path joining $u$ and $v$ in $G$. The diameter of $G$, denoted by $diam(G)$, is $\max\{d_G(u,v): u, v\in V(G)\}$. The symbols $P_n, C_n, K_n$ represent the path, cycle, complete graph  of order $n$, respectively. A bipartite graph with bipartition $(X, Y)$ is denoted by $G[X,Y]$. In particular, $K_{m,n}$ denotes the complete bipartite graph with $m$ and $n$ vertices in its two parts. For a bipartite graph $G=G[X,Y]$, the {\it bipartite complement} $\overline{G_B}$ of $G$ is the simple bipartite graph with bipartition $(X,Y)$ in which $x$ and $y$ are adjacent if and only if they are not adjacent in $G$.

Gallai (see \cite{Lovasz1979}, Section 5, Problem 17) proved that the vertices of any graph can be partitioned into two sets, each of which induces a subgraph with all degrees even; the vertices also can be partitioned into two sets so that one set induces a subgraph with all degrees even and the other induces a subgraph with all degrees odd. As an immediate consequence of this, we see that every graph of order $n$ contains an induced subgraph of order at least $\lceil\frac n 2\rceil$ with all degrees even.

With motivation from this fact, we are interested in the problem that how large can be the order of an induced subgraph with all degrees odd in a given graph. For a graph $G$, $f_{o}(G)$ is the maximum order of an odd induced subgraph of a graph $G$. Clearly, $f_o(G)$ is even for any graph $G$ by the hand-shaking lemma. More results related to $f_o(G)$ can be found in \cite{Belmonte2021, Berman1997, Caro1994, Gutin2016, Radcliff1995, Scott1992, Scott2001}.
The following long-standing conjecture was first appeared in \cite{Caro19942}.

\begin{conjecture} [\cite{Caro19942}] There exists a positive constant $c$ such that every $n$-vertex graph without isolated vertices contains an odd induced subgraph with at least $cn$ vertices.
\end{conjecture}

Moreover, much works are devoted to confirm the validity of the above conjecture for some special families of graphs, such as subcubic graphs \cite{Berman}, graphs with treewidth at most 2 \cite{Hou2018}, trees \cite{Radcliff1995} and planar graphs \cite{Rao2022}.
Very recently, Ferber and Krivelevich \cite{Ferber2021} proved that Conjecture 1.1 is true by taking $c=10^{-4}$.

\section{\large Counterexamples on bipartite graphs}
Let $\chi(G)$ be the chromatic number of a graph $G$. In 1992, Scott \cite{Scott1992} showed that, $f_{o}(G)\geq \frac{n}{2\chi(G)}$ for any graph $G$ of order $n$ without isolated vertices, and proposed the following conjecture.

\begin{conjecture} [Scott \cite{Scott1992}] For any graph $G$ of order $n$ without isolated vertices, $$f_{o}(G)\geq \frac{n}{\chi(G)}.$$
\end{conjecture}

We disprove this for bipartite graphs. 
First recall a few definitions \cite{Bondy}. A hypergraph $(P, L)$ is called a {\it geometric configuration} at most one element of $L$ contains any given pair of elements. An element of $P$ are called a {\it point} and  an element of $L$ are called a {\it line}. 
A {\it finite projective plane} is a geometric configuration in which 

(i) any two points lie exactly one line, 

(ii) any two lines meet in exactly one point,

(iii) there are four points no three of which lie on a line. 

\vspace{2mm} Fano plane is a projective plane with $P=\{1, 2, \ldots, 7\}$ and $L=\{124,235,346,\\457,156, 567,137\}$, as shown in Fig 1.
\begin{center}
\scalebox{0.30}[0.30]{\includegraphics{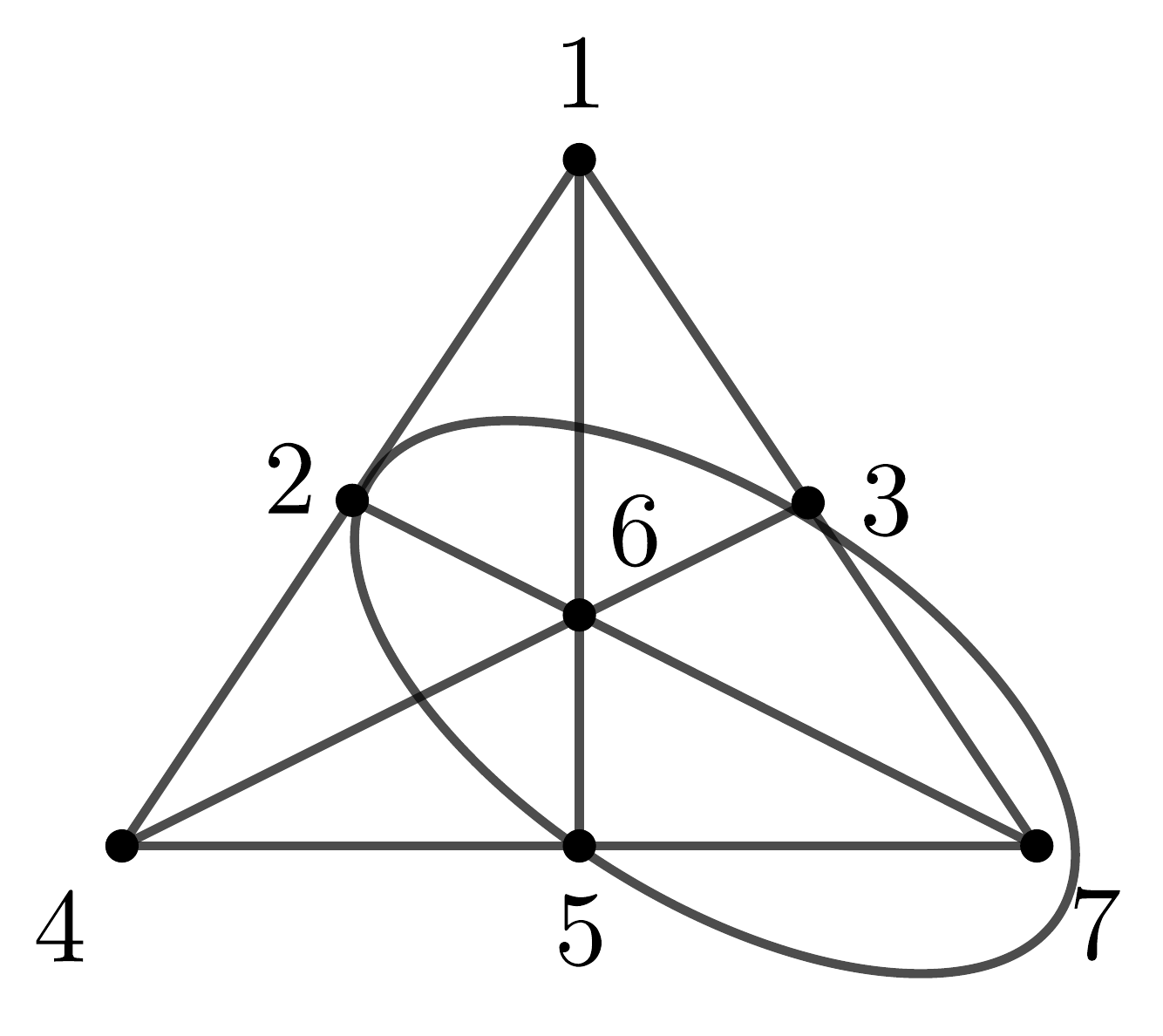}}
\centerline{Fig. 1.~Fano plane}
\end{center}

The incidence graph of Fano plane, known as the Heawood graph, is the bipartite graph in Fig. 2 (not containing those edges colored in red). Let $G$ be the graph obtained from Heawood graph by adding a perfect matching as depicted in red.
Formally, $V(G)=X\cup Y$, where $X=\{1,2,3,4,5,6,7\}, Y=\{1247,1235,2346,3457,1456, 2567,1367\}$. Two vertices $x\in X$ and $y\in Y$ are adjacent if and only if $x\in y$.

Surprisingly, $G$ is isomorphic to the bipartite complement graph of the Heawood graph $H$.
We will see that $G$ is a counterexample to Conjecture 2.1, because of $f_{o}(G)=6<\frac{14}{2}=\frac{v(G)} {\chi(G)}$.


\begin{center}
\scalebox{0.35}[0.35]{\includegraphics{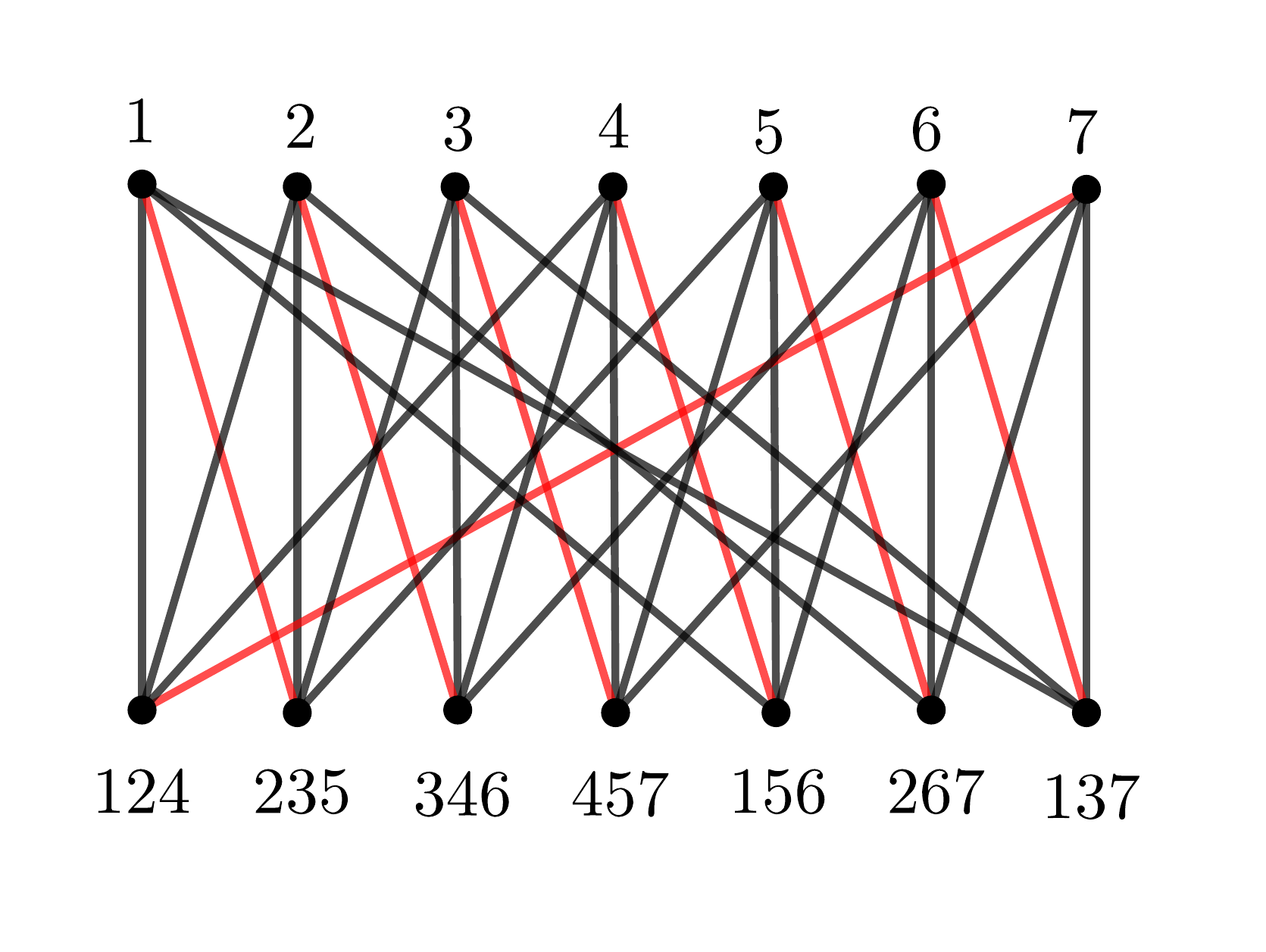}}
\centerline{Fig. 2.~Heawood graph $H$ with an extra perfect matching $M$ and $G=H+M$}
\end{center}

\begin{center}
\scalebox{0.30}[0.30]{\includegraphics{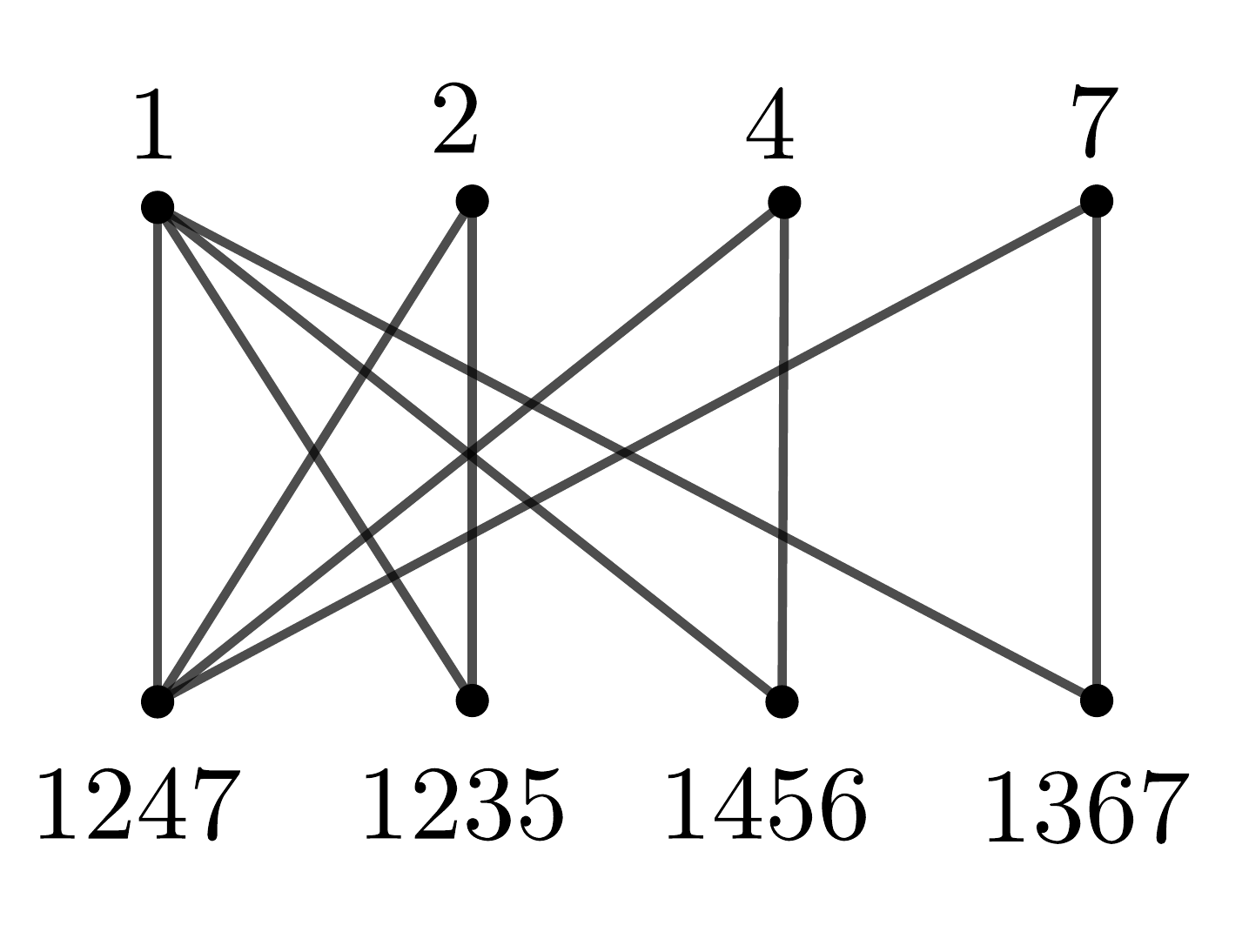}}
\centerline{Fig. 3.~$G[N(x)\cup N(y)]$ for an edge $xy\in E(G)$ }
\end{center}

\begin{lemma}
Let $H$ be the Heawood graph and $G=H+M$ as shown in Fig. 2, where $M$ is the extra perfect matching depicted in red. The following statements hold:

(1) $G$ is a vertex-transitive and edge-transitive 4-regular graph.

(2) $G[N(x)\cup N(y)]$ is isomorphic to the graph shown in Fig. 3 for an edge $xy\in E(G)$. 

(3) $|N_G(x)\cap N_G(x')|=2$ for any two elements $x, x'\in X$; Since $G$ is transitive, $|N_G(y)\cap N_G(y')|=2$ for any two elements $y, y'\in Y$.

(4) $G$ has no subgraph isomorphic to $K_{2,3}$. 
\end{lemma}

\begin{proof}
(1) is immediate from the observation that $G$ is isomorphic to the bipartite complement of $H$ and $H$ is a vertex-transitive and edge-transitive 3-regular graph.

By the definition (i) of a projective plane, $|N_H(x)\cap N_H(x')|=1$ for any two vertices $x, x'\in X$. Since $|X|=7$, there are exactly two vertices of $Y$ adjacent to neither $x$ nor $x'$ in $H$, which are, in turn, adjacent to both $x$ and $x'$ in $G$. This proves (2). 

(3) is a direct consequence of (1), and $(4)$ can be derived from (3).

\end{proof}

\begin{theorem}
If $G$ is the bipartite complement of the Heawood graph, then $f_{o}(G)=6<\frac{14}{2}$.
\end{theorem}
\begin{proof}
For convenience, let $V(G)=(X,Y)$, where $X=\{1,2,3,4,5,6,7\}$, $Y=\{1247,1235, 2346,3457,1456,2567,1367\}$. Since $\{2, 4, 7, 1235, 1456, 1367\}$ induces an odd induced subgraph with all degrees 1, $f_o(G)\geq 6$. It remains to show that $f_o(G)\leq 6$.
Let $F$ be an odd induced subgraph of $G$ induced by $S$.
Since $G$ is bipartite, so is $F$. Since $G$ is vertex-transitive, we may assume that $|S\cap X|\leq |S\cap Y|$, without loss of generality. We shall show that $|S|\leq 6$ in terms of $|S\cap X|$. 

\vspace{2mm}\noindent{\bf Case 1.} $|S\cap X|=1$

\vspace{2mm} Let $S\cap X=\{x\}$. Clearly, $|S\cap Y|\leq d_F(x)$. Since $d_F(x)$ is odd and $d_F(x)\leq d_G(x)=4$, $d_F(x)\leq 3$. Hence $|S|\leq 4$.

\vspace{2mm}\noindent{\bf Case 2.} $|S\cap X|=2$ 

\vspace{2mm} Let $S=\{x, x'\}$.
Since $|N_G(x)\cap N_G(x')|=2$ for any two elements $x, x'\in X$, it follows that $|S\cap Y|\leq 4$, and thus $|S|\leq 6$.

\vspace{2mm} \noindent{\bf Case 3.} $|S\cap X|=3$ 

\vspace{2mm} Let $S\cap X=\{a,b,c\}$. 

\vspace{2mm} \noindent{\bf Subcase 3.1.} There exists a vertex $y\in Y$ adjacent to all elements of $S\cap X$ in $G$. 

\vspace{2mm} Let $y_{ab}$ be the vertex of $Y$ adjacent to both $a$ and $b$ in $G$ other than $y$. Furthermore, $y_{ac}$ and $y_{bc}$ are the vertices of $Y$ are defined similarly. Each of $y_{ab}$, $y_{ac}$ and $y_{bc}$ has exactly two neighbors in $S\cap X$ and hence, cannot belong to $S$. Thus $|S\cap Y|\leq 4$. Combining this with the fact that $|S\cap Y|$ is odd, $|S\cap Y|\leq 3$, and hence $|S|\leq 6$.

\vspace{2mm}\noindent{\bf Subcase 3.2.} There exists no vertex $y\in Y$ adjacent to all elements of $S\cap X$ in $G$. 

\vspace{2mm} It means that if a vertex $y$ is adjacent to two vertices of $S\cap X$ in $G$, then $y\notin S$. However, there are six such vertices in all by Lemma 2.2 (3), implying $|S\cap Y|\leq 1$. This contradicts $|S\cap Y|\geq 3$.  

\vspace{2mm} In the remaining proof, we show that $|S\cap X|\geq 4$ is impossible. 

\vspace{2mm}\noindent{\bf Case 4.} $|S\cap X|=4$ 

\vspace{2mm} By $|S\cap X|\leq |S\cap Y|$ and $|S|$ is even, we have $|S\cap Y|\in\{4, 6\}$.

\vspace{2mm}\noindent{\bf Subcase 4.1.} $|S\cap Y|=6$

\vspace{2mm} Since $G$ is vertex-transitive, we may assume that $Y\setminus S=\{y\}$, where $y=1367$, without loss of generality. Note that each element $x$ of $\{2,4,5\}$ cannot belong to $S$, because of $|N_G(x)\cap S|=4$. Moreover, since $|X|=7$, we have $X\cap S=X\setminus \{2, 4, 5\}=\{1, 3, 6, 7\}$. However, 1247 has exactly two neighbors 1 and 7 in $X\cap S$. This contradicts that $d_F(1247)$ is odd. 

\vspace{2mm}\noindent{\bf Subcase 4.2.} $|S\cap Y|=4$

\vspace{2mm} Let $S\cap X=\{a,b,c,d\}$. Observe that $F$ is an odd graph of order 8 with $d_F(v)\in\{1, 3\}$. Since $F$ is bipartite, $|\{x: x\in S\cap X \ \text{with}\ d_F(x)=3\}|=|\{y: y\in S\cap Y\ \text{with}\ d_F(y)=3\}|$ and $|\{x: x\in S\cap X \ \text{with}\  d_F(x)=1\}|=|\{y: y\in S\cap X \ \text{with}\  d_F(y)=1\}|$. Let $n_i=|\{x: x\in S\cap X \ \text{with}\ d_F(x)=i\}|$ for $i\in\{1, 3\}$. 

%

\vspace{2mm}\noindent{\bf Subcase 4.2.1.} $n_3=4$

\vspace{2mm} It implies that $e(F)=4\times 3=12$ and thus $e_G(S, V(G)\setminus S)=8\times 4-2\times 12=8$. Therefore, $G-S$ has 6 vertices and $e(G-S)=\frac {6\times 4-8} 2=8$, Moreover, since $G-S$ is bipartite graph with 3 vertices each part, $G-S\cong K_{3,3}-e$. This contradicts the assertion (4) of Lemma 2.2.

\vspace{2mm}\noindent{\bf Subcase 4.2.2.} $n_3=3$

\vspace{2mm}Let $x_1$ and $y_1$ are the two vertices of $F$ with degree 1. Thus, $F-x_1-y_1\cong K_{3,3}-e$. On the other hand, by Lemma 2.2 (4) $F-x_1-y_1$ is isomorphic to the subgraph $G[N(x)\cup N(y)]$ deleting two vertices 7 and 1367. Clearly, this is impossible.  

\vspace{2mm}\noindent{\bf Subcase 4.2.3.}  $n_3=0$

\vspace{2mm} It follows that $e_G(S\cap X, Y\setminus S)=e_G(S\cap Y, X\setminus S)=4\times 3$. On the other hand, since $\sum_{x\in X\setminus S} d_G(x)=4\times |X\setminus S|=12=4\times|Y\setminus S|=\sum_{y\in Y\setminus S} d_G(y)$. Summing the above two facts, we conclude that $G-S$ is edgeless, i.e.  
$H-S\cong K_{3,3}$. This is impossible, $H$ is 3-regular, $H-S$ is a proper subgraph of $H$.

\vspace{2mm}\noindent{\bf Subcase 4.2.4.} $n_3=1$

\vspace{2mm} The degree sequence of $F$ must be $3, 3, 1, 1, 1, 1, 1, 1$. Since a graph containing a cycle has at least three vertices of degree greater than 1, $F$ is acyclic. 
Let $x\in S\cap X$ and $y\in S\cap Y$ be the two vertices with degree 3 in $F$. 
If $xy \notin E(F)$, then $F\cong 2K_{1,3}$ ($F$ consists of two disjoint claws). Since $G$ is vertex-transitive and edge-transitive, we may assume that $1\in S\cap X$ and $N_F(1)=\{1247, 1235, 1456\}$, without loss of generality. Clearly, each vertex of $X$ is adjacent to at least one element of $\{1247, 1235, 1456\}$ in $G$, and therefore $S\cap X\subseteq \{1\}$, contradicting $|S\cap X|=4$.
If $xy \in E(F)$, then $F\cong G[N(x)\cup N(y)]-x'-y'$  for some vertices $x'\in N_G(y)\setminus \{x\}$ and $y'\in N_G(x)\setminus \{y\}$. However, by Lemma 2.2 (2), $G[N(x)\cup N(y)]-x'-y'$ always contain a cycle. Again, this is a contradiction.

\vspace{2mm}\noindent{\bf Subcase 4.2.5.} $n_3=2$

\vspace{2mm} Since $|S\cap X|=|S\cap Y|=4$ and $n_3=2$, the four vertices of $F$ with degree 3 induces a $C_4$. Since $G$ is vertex-transitive, without loss of generality, we may assume that $V(C_4)=\{1, 2, 1247, 1235\}$. Clearly, these four vertices have degree 3 in $F$. Moreover, each element of $V(F)\setminus V(C_4)$ has 1 in $F$ and $G$ is edge-transitive, we may assume that $\{4, 1367\}\subseteq V(F)$.  
Since $d_F(1)=1$ and $1247\in V(F)$, we have $2346, 3457, 1456\notin V(F)$, and thus $2567\in V(F)$. 
Moreover, since $d_F(1235)=3$, $3\in V(F)$ or $5\in V(F)$. However, neither $3\in V(F)$ nor $5\in V(F)$, because 3 has two neighbor 1235 and $1367$, 5 has two neighbors 1235 and 2567. Thus, $F$ cannot exists. 

\vspace{2mm}\noindent{\bf Case 5.} $|S\cap X|=5$

\vspace{2mm} Since $|S|$ is even, $|S\cap Y|\in\{5, 7\}$.
Let $X\setminus S=\{a, b\}$. Let $y_{ab}^1$ and $y_{ab}^2$ be the two vertices in $Y$ adjacent to $a$ and $b$, and $y_{cdef}$ be the unique element containing neither $a$ nor $b$.
Clearly, $y_{ab}^1$ and $y_{ab}^2$ have degree two and $y_{cdef}$ has degree 4 in $G-a-b$. Thus, $y_{ab}^1\notin S\cap Y$, $y_{ab}^2\notin S\cap Y$ and $y_{cdef}\notin S\cap Y$, implying that $|S\cap Y|\leq 4$. This contradicts the assumption $|S\cap Y|\geq |S\cap X|\geq 5$.

\vspace{2mm}\noindent{\bf Case 6.} $|S\cap X|=6$

\vspace{2mm} But, this is impossible.  
Suppose that $X\setminus S=\{x\}$ and $Y\setminus S=\{y\}$. Since $d_G(x)=4$, there exists a vertex $y'\in Y\setminus \{y\}$ which is not adjacent to $x$ in $G$. Clearly, $d_F(y')=4$ contradicts that $d_F(y')$ is odd.  
\end{proof}

Remark that $G-v$ is also a counterexample to Conjecture 2.1. It might be possible that one may find more counterexamples by using finite projective planes. Scott's conjecture is still open for graphs with chromatic number at least 3.
We may obtain an infinite number of counterexamples to Scott's conjecture by the following way.

\begin{corollary}
If $kG$ is the disjoint union of $k$ copies of $G$ for any positive integer $k$, then $f_o(kG)=kf_o(G)<k\frac {v(G)} 2=\frac {v(kG)} {\chi(kG)}$.
\end{corollary}

\section{\large Line graphs}

For a graph $G$, the {\it line graph} of $G$, denoted by $L(G)$, is the graph with $V(L(G))= E(G)$, in which two vertices are adjacent if and only if they are adjacent as edges of $G$.

A subgraph $H$ of a graph $G$ is called an {\it odd-even} edge induced subgraph if all adjacent vertices in $H$ have degrees with distinct parity. Note that every odd-even edge induced graph is a bipartite graph with a bipartition in which one part consists of vertices of degree odd, while the other part consists of vertices of degree even.
Observe that every odd induced subgraph of the line graph $L(G)$ of a graph $G$ corresponds to an odd-even edge induced subgraph of $G$. We use $\varepsilon(G)$ denote the maximum size of an odd-even edge induced subgraph of $G$.
Clearly, $\varepsilon(G)=f_0(L(G))$ for any graph $G$.

%
%

\begin{theorem}\label{tree}
If $T$ is a tree of order $n$, then $f_o(L(T))\geq f(n)$, where
$$
f(n)=\left \{
\begin{array}{ll}
\lceil\frac n 2\rceil, & \mbox{if $n\equiv 0,3$ $(\hspace{-4mm}\mod 4)$}\\
\lfloor\frac {n-1} 2\rfloor, &\mbox{if $n\equiv 1,2$ $(\hspace{-4mm}\mod 4)$.}
\end{array}
\right.
$$
\end{theorem}
\begin{proof} Let $T$ be a tree of order $n$. It is enough to show that $\varepsilon(T)\geq f(n)$.
Note that
$$
\varepsilon(K_{1,n-1})=\left \{
\begin{array}{ll}
n-1, & \mbox{if $n\equiv 1$ $(\hspace{-4mm}\mod 2)$}\\
n-2, &\mbox{if $n\equiv 0$ $(\hspace{-4mm}\mod 2)$.}
\end{array}
\right.
$$
Also, $\varepsilon(P_4)=2=f(4)$. So, now let $n\geq 5$. For $e=uv\in E(T)$, let $T_u$ and $T_v$ be the two components of $T-uv$ containing $u$ and $v$, respectively. For simplicity, $n_u=|V(T_u)|$ and $n_v=|V(T_v)|$. We proceed with the induction on $n$. We divide four cases in terms of order.

\vspace{3mm}\noindent{\bf Case 1.} There exists an edge $uv\in E(T)$ such that $n_u\equiv 0$ or $3$ (mod 4), or $n_v\equiv 0$ or $3$ (mod 4).

\vspace{2mm} By symmetry, it suffices to consider the following two possibilities.

\vspace{2mm}\noindent{\bf Subcase 1.1.} $n_u\equiv 0$ (mod 4)

\vspace{2mm} Clearly, $n\equiv n_v$ (mod 4). Moreover, by the induction hypothesis, $\varepsilon(T_u)\geq \frac{n_u}{2}$ and $\varepsilon(T_v)\geq f(n_v)$. Thus, $\varepsilon(T)\geq \frac{n_u}{2}+f(n_v)=f(n)$.

\vspace{2mm}\noindent{\bf Subcase 1.2.} $n_u\equiv 3$ (mod 4)

Since
$$
n= \left \{
\begin{array}{ll}
3\ (\hspace{-4mm}\mod 4), & \mbox{if $n_v\equiv 0$ (mod 4) }\\
0\ (\hspace{-4mm}\mod 4), & \mbox{if $n_v\equiv 1$ (mod 4) }\\
1\ (\hspace{-4mm}\mod 4), & \mbox{if $n_v\equiv 2$ (mod 4) }\\
2\ (\hspace{-4mm}\mod 4), &\mbox{if $n_v\equiv 3$ (mod 4), }
\end{array}
\right.
$$
and by the induction hypothesis $\varepsilon(T_u)\geq \frac{n_u+1}{2}$,
we have $\varepsilon(T)\geq \frac{n_u+1}{2}+f(n_v)\geq f(n)$.

\vspace{3mm}\noindent{\bf Case 2.} There exists an edge $uv\in E(T)$ such that $n_u\equiv 2$ (mod 4) and $n_v\equiv 2$ (mod 4).

\vspace{2mm} By the assumption, $n\equiv 0$ (mod 4). Let $T'=T-x$, where $x$ is a leaf. Obviously, $v(T')=n-1\equiv 3$ (mod 4). Moreover, by the induction hypothesis, $\varepsilon(T')\geq \frac{v(T')+1}{2}$. Thus, $\varepsilon(T)\geq \varepsilon(T')\geq \frac{v(T')+1}{2}=\frac {n} 2$.

\vspace{3mm}\noindent{\bf Case 3.} There exists an edge $uv\in E(T)$ such that $n_u\equiv 1$ (mod 4) and $n_v\equiv 1$ (mod 4).

\vspace{2mm} By the assumption, $n\equiv 2$ (mod 4). Take such an edge $uv$ as described in the assumption. By the induction hypothesis, $\varepsilon(T_u)\geq \frac{n_u-1}{2}$, $\varepsilon(T_v)\geq \frac{n_v-1}{2}$, it follows that $\varepsilon(T)\geq \frac{n_u+n_v-2}{2}=\frac {n-2} 2$.

\vspace{3mm}\noindent{\bf Case 4.} $n_u\equiv 1$ (mod 4) and $n_v\equiv 2$ (mod 4) for any edge $uv\in E(T)$.

\vspace{2mm}  It is clear that $n\equiv 3$ (mod 4).
By the induction hypothesis, $\varepsilon(T_u)\geq \frac{n_u-1} 2$ $\varepsilon(T_v)\geq \frac{n_v-2} 2$. If there exists an edge $uv$ such that $\varepsilon(T_u)>\frac{n_u-1} 2$ or $\varepsilon(T_v)>\frac{n_v-2} 2$, then we are done. So, in what follows, we assume that for any edge $uv$,
\begin{equation}\varepsilon(T_u)=\frac{n_u-1} 2\ \text{and}\ \varepsilon(T_v)=\frac{n_v-2} 2.\end{equation}

Consider a vertex $x\in V(T)$. Let $T_1, \ldots, T_l$ be all components of $T-x$, where $l=d_T(x)$. Let $n_i$ be the order of $T_i$ and $x_i$ be the neighbor of $x$ in $T_i$ for each $i\in\{1, \ldots, l\}$. Let $T_i'=T_i+xx_i$ for each $i$. Without loss of generality, let $a$ be the integer between 0 and $l$ such that 
$$
n_i\equiv \left \{
\begin{array}{ll}
1~(\hspace{-2.5mm}\mod 4), & \mbox{if $i\leq a$  }\\
2~(\hspace{-2.5mm}\mod 4), & \mbox{if $i\geq a+1$.}\\
\end{array}
\right.
$$
Let $b=l-a$. Since $n\equiv a+2b+1\equiv 3 ~(\hspace{-2.5mm}\mod 4)$, $a+2b\equiv 2~ (\hspace{-2.5mm}\mod 4)$. Thus, $a$ is even.
Furthermore, if $a\equiv 2 ~(\hspace{-2.5mm}\mod 4)$, then $b$ is even, and if  $a\equiv 0~(\hspace{-2.5mm}\mod 4)$, then $b$ is odd.

\begin{center}
\scalebox{0.30}[0.30]{\includegraphics{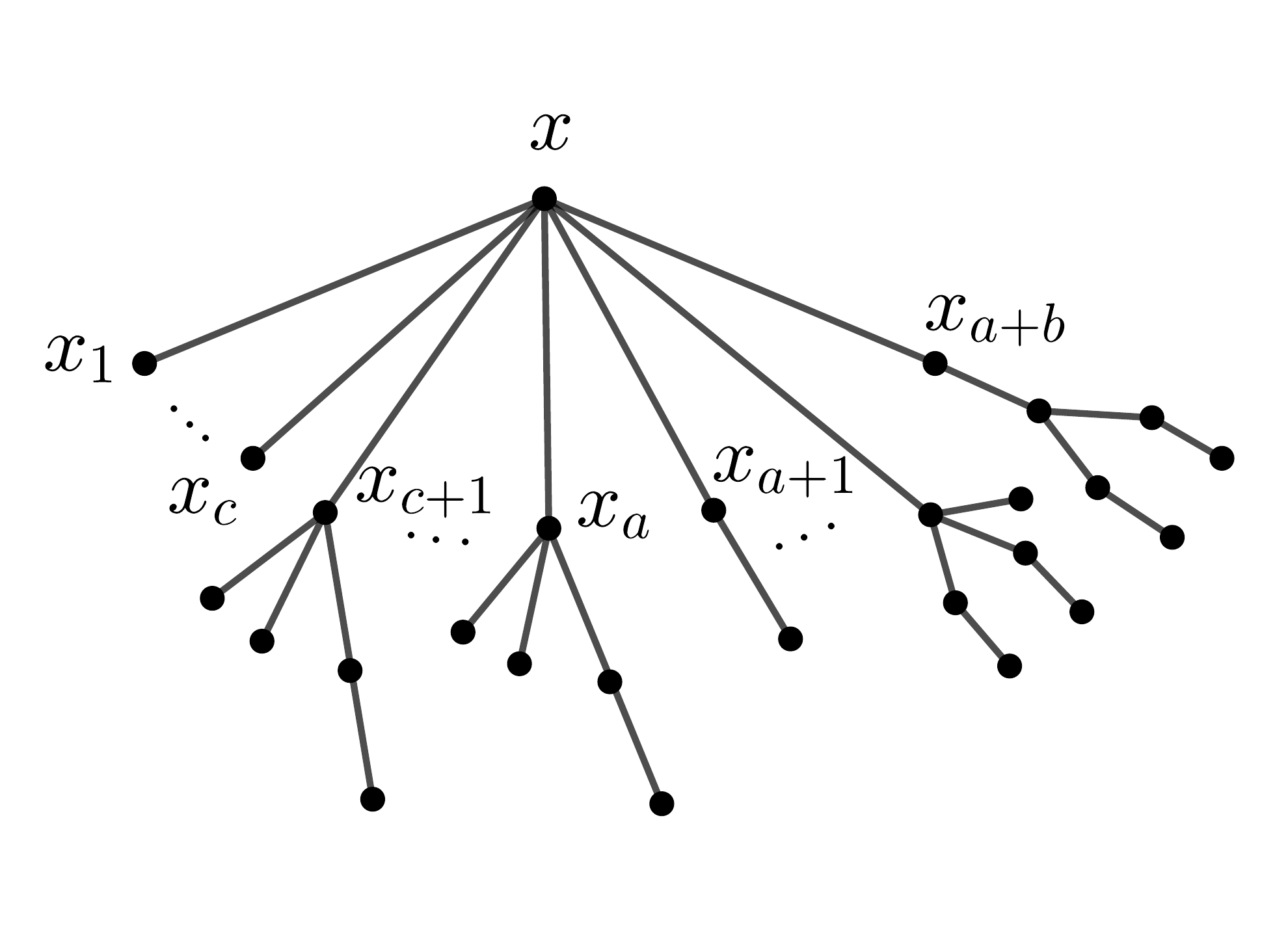}}
\centerline{Fig. 4.~The components of $T-x$}
\end{center}

Let $i, j\in \{1,2,\ldots,a\}$.
Since $T_i'\cup T_j'$ is a tree of order congruent to $3~(\hspace{-2.5mm}\mod 4)$, by the induction hypothesis, $\varepsilon(T_i'\cup T_j')\geq \frac{n_i+n_j+2} 2$. 
On the other hand, by (1) $\varepsilon(T_i)=\frac{n_i-1} 2$ and $\varepsilon(T_j)=\frac{n_j-1} 2$. 
Clearly, $1\leq d_{B_{ij}}(x)\leq 2$ for an odd-even edge induced subgraph $B_{ij}$ of size $\varepsilon(T_i'\cup T_j')$ in $T_i'\cup T_j'$. If $d_{B_{ij}}(x)=2$, then both $T_i$ and $T_j$ are said to be {\it pairable}.  If $d_{B_{ij}}(x)=1$, then $d_{B_{ij}}(x_i)=1$ or $d_{B_{ij}}(x_j)=1$. The component $T_i$ with $d_{B_{ij}}(x_i)=1$ (or $T_j$ with $d_{B_{ij}}(x_j)=1$) is called {\it impairable}.  
Note that $T_i$ is pairable if and only if $\varepsilon(T_i')=\frac{n_i-1} 2$ and $T_i$ is impairable if and only if $\varepsilon(T_i')\geq \frac{n_i+3} 2$. 
Thus, whether a component $T_i$ being pairable or not does not depend upon the choice of $T_j$.
Let $c$ be the number of pairable components, and $d=a-c$. Since $a$ is even, $c$ and $d$ have the same parity. Without loss of generality, assume that $T_i$ is pairable for each $i\in\{1, \ldots, c\}$ and the remaining ones are impairable. One can see that $x_i$ is pairable for each $i\leq c$ and the component containing $x_d$ is impairable, as illustrated in Fig. 4.

We distinguish the following cases.

\vspace{3mm}\noindent{\bf Subcase 4.1.} $a\equiv 2~ (\hspace{-2.5mm}\mod 4)$

\vspace{2mm} Recall that $b$ is even.

\vspace{2mm}\noindent{\bf Subcase 4.1.1.} Both $c$ and $d$ are odd.
\begin{eqnarray*}
\varepsilon(T)&\geq&\sum_{i=1}^c \frac{n_i-1} 2+\sum_{i=c+1}^{c+d} \frac{n_i+3} 2+ \sum_{i=a+1}^{a+b} \frac{n_i+2} 2\\
&=& \sum_{i=1}^l \frac{n_i} 2 +\frac{3d+2b-c} 2 \\
&=& \frac{n+1} 2+\frac{3d+2b-c-2} 2.
\end{eqnarray*}

It follows that if $c\leq 3d+2b-2$, then $\varepsilon(T)\geq \frac {n+1} 2$.
Since both $c$ and $d$ are odd, it remains to consider the case when \begin{equation} c\geq 3d+2b. \end{equation}
\begin{eqnarray*}
\varepsilon(T)&\geq&\sum_{i=1}^{\frac{c-1} 2} \frac{n_i+n_{c-i}+2} 2+\sum_{i=c}^{c+d} \frac{n_i-1} 2+ \sum_{i=a+1}^{a+b} \frac{n_i-2} 2\\
&=& \sum_{i=1}^l \frac{n_i} 2 +\frac{(c-1)-(d+1)-2b} 2 \\
&=& \frac{n+1} 2+\frac{c-d-2b-4} 2.
\end{eqnarray*}

It follows that if $c\geq d+2b+4$, then $\varepsilon(T)\geq \frac {n+1} 2$.
Since both $c$ and $d$ are odd, next assume that \begin{equation} c\leq d+2b+2. \end{equation}

Combining (2) and (3), $3d+2b\leq c\leq d+2b+2$. It follows that $d=1$ and $c=2b+3$. However, $a=c+d=2b+4\equiv 0 ~(\hspace{-2.5mm}\mod 4)$, contradicting the assumption that $a\equiv 2 ~(\hspace{-2.5mm}\mod 4)$.

\vspace{2mm}\noindent{\bf Subcase 4.1.2.} Both $c$ and $d$ are even.
\begin{eqnarray*}
\varepsilon(T)&\geq&\sum_{i=1}^c \frac{n_i-1} 2+\sum_{i=c+1}^{c+d} \frac{n_i+3} 2+ \sum_{i=a+1}^{a+b-1} \frac{n_i+2} 2+\frac{n_{a+b}-2} 2\\
&=& \sum_{i=1}^l \frac{n_i} 2 +\frac{3d+2(b-1)-c-2} 2 \\
&=& \frac{n+1} 2+\frac{3d+2b-c-6} 2.
\end{eqnarray*}

It follows that if $c\leq 3d+2b-6$, then $\varepsilon(T)\geq \frac {n+1} 2$.
Since both $c$ and $d$ are even, it remains to consider the case when \begin{equation} 3d+2b-4\leq c. \end{equation}

\begin{eqnarray*}
\varepsilon(T)&\geq&\sum_{i=1}^{\frac c 2} \frac{n_i+n_{c+1-i}+2} 2+\sum_{i=c}^{c+d} \frac{n_i-1} 2+ \sum_{i=a+1}^{a+b} \frac{n_i-2} 2\\
&=& \sum_{i=1}^l \frac{n_i} 2 +\frac{c-d-2b} 2 \\
&=& \frac{n+1} 2+\frac{c-d-2b-2} 2.
\end{eqnarray*}

It follows that if $c\geq d+2b+2$, then $\varepsilon(T)\geq \frac {n+1} 2$.

Since both $c$ and $d$ are even, we assume that \begin{equation} c\leq d+2b. \end{equation}

Combining (4) and (5), we have $$2b+3d-4\leq c\leq 2b+d,$$ implying $d=0$ or $d=2$.

\vspace{2mm} If $d=2$, then $c=2b+2$. But, $a=c+d=2b+4\equiv 0 ~(\hspace{-2.5mm}\mod 4)$ contradicts the assumption $a\equiv 2~(\hspace{-2.5mm}\mod 4)$. If $d=0$, then $2b-4\leq c\leq 2b.$ Combining this with the facts that $c=a$ and $b$ is even, we have $c=2b-2$.

\vspace{3mm}\noindent{\bf Subcase 4.2.} $a\equiv 0~ (\hspace{-2.5mm}\mod 4)$

\vspace{2mm} Recall that $b$ is odd.

\vspace{2mm}\noindent{\bf Subcase 4.2.1.} Both $c$ and $d$ are even.
\begin{eqnarray*}
\varepsilon(T)&\geq&\sum_{i=1}^c \frac{n_i-1} 2+\sum_{i=c+1}^{c+d} \frac{n_i+3} 2+ \sum_{i=a+1}^{a+b} \frac{n_i+2} 2\\
&=& \sum_{i=1}^l \frac{n_i} 2 +\frac{3d+2b-c} 2 \\
&=& \frac{n+1} 2+\frac{3d+2b-c-2} 2.
\end{eqnarray*}

It follows that if $3d+2b-2\geq c$, then $\varepsilon(T)\geq \frac {n+1} 2$.
Since both $c$ and $d$ are even, it remains to consider the case when \begin{equation} 3d+2b\leq c. \end{equation}

\begin{eqnarray*}
\varepsilon(T)&\geq&\sum_{i=1}^{\frac{c} 2} \frac{n_i+n_{c+1-i}+2} 2+\sum_{i=c}^{c+d} \frac{n_i-1} 2+ \sum_{i=a+1}^{a+b} \frac{n_i-2} 2\\
&=& \sum_{i=1}^l \frac{n_i} 2 +\frac{c-d-2b} 2 \\
&=& \frac{n+1} 2+\frac{c-d-2b-2} 2.
\end{eqnarray*}

It follows that if $c\geq d+2b+2$, then $\varepsilon(T)\geq \frac {n+1} 2$.
Since both $c$ and $d$ are even, next assume that \begin{equation} c\leq d+2b. \end{equation}

Combining (6) and (7), $3d+2b\leq c\leq d+2b$. Thus $d=0$ and $c=2b$. But, $a=2b\equiv 2~ (\hspace{-2.5mm}\mod 4)$, contradicts the assumption $a\equiv 0 ~(\hspace{-2.5mm}\mod 4)$.

\vspace{2mm}\noindent{\bf Subcase 4.2.2.} Both $c$ and $d$ are odd.
\begin{eqnarray*}
\varepsilon(T)&\geq&\sum_{i=1}^c \frac{n_i-1} 2+\sum_{i=c+1}^{c+d} \frac{n_i+3} 2+ \sum_{i=a+1}^{a+b-1} \frac{n_i+2} 2+\frac{n_{a+b}-2} 2\\
&=& \sum_{i=1}^l \frac{n_i} 2 +\frac{3d+2(b-1)-c-2} 2 \\
&=& \frac{n+1} 2+\frac{3d+2b-c-6} 2.
\end{eqnarray*}

It follows that if $c\leq 3d+2b-6$, then $\varepsilon(T)\geq \frac {n+1} 2$.
Since both $c$ and $d$ are odd, next assume that \begin{equation} c\geq 3d+2b-4. \end{equation}

\begin{eqnarray*}
\varepsilon(T)&\geq&\sum_{i=1}^{\frac {c-1} 2} \frac{n_i+n_{c-i}+2} 2+\sum_{i=c}^{c+d} \frac{n_i-1} 2+ \sum_{i=a+1}^{a+b} \frac{n_i-2} 2\\
&=& \sum_{i=1}^l \frac{n_i} 2 +\frac{(c-1)-(d+1)-2b} 2 \\
&=& \frac{n+1} 2+\frac{c-d-2b-4} 2.
\end{eqnarray*}

It follows that if $c\geq d+2b+4$, then $\varepsilon(T)\geq \frac {n+1} 2$.
Since both $c$ and $d$ are odd, next assume that \begin{equation} c\leq d+2b+2. \end{equation}

Combining (8) and (9), we have $2b+3d-4\leq c\leq 2b+d,$ implying $d=1$ or $d=3$. If $d=3$, then $c=2b+5$. But, $a=c+d=2b+8\equiv 2 ~(\hspace{-2.5mm}\mod 4)$ contradicts the assumption $a\equiv 0 ~(\hspace{-2.5mm}\mod 4)$.
If $d=1$, then $2b-1\leq c\leq 2b+3$. Since $a=c+d=c+1$, $a\equiv 0 ~(\hspace{-2.5mm}\mod 4)$ and $b\equiv 1$ or $3 (\hspace{-2mm}\mod 4)$, we have $a=2b+2$, and thus $c=2b+1$.

\vspace{3mm} Summing up the above, we conclude that
{\it if $\varepsilon(T)\leq \frac{n-3} 2$, then

(1) $d=0$, $c=2b-2$, $b>0$ is even, or $d=1$, $c=2b+1$, $b>0$ is odd;

(2) $\varepsilon(T_u)=\frac{n_u-1} 2$ and $\varepsilon(T_v)=\frac{n_u-2} 2$ for any $uv\in E(T)$.}

\vspace{3mm}\noindent {\bf Claim 1.} There exists a vertex $v$ such that all components of $T-v$ with order 2 ~(\hspace{-2.5mm}$\mod 4$) have order 2.

\begin{proof} By the remark above, $b>0$ for any vertex $x$. Take a vertex, say $v_1$. If $T-v_1$ satisfies our requirement, we are done. Otherwise, there is a component of $T-v_1$ having order 2 (mod 4) greater than 2. Let $v_2$ be the neighbor of $v_1$ in that component. Consider $T-v_2$. Again, let us check whether $T-v_2$ meets our requirement or not. If it is, let us stop. Otherwise, we repeat the above procedure and end up with a vertex $v$ as we want.
\end{proof}

Now let $x$ be the vertex of $T$ as we claimed. First, assume that $d=0$, $c=2b-2$, $b>0$ is even.
\begin{eqnarray*}
\varepsilon(T)&\geq&\sum_{i=1}^{\frac {c} 2} \frac{n_i+n_{c+1-i}+2} 2+b\\
&=& \sum_{i=1}^c \frac{n_i} 2+ c+ b \\
&=& \frac{n+1} 2+ c-\frac 1 2\\
&>& \frac{n+1} 2.
\end{eqnarray*}

Secondly, assume that $d=1$, $c=2b+1$, $b>0$ is odd.
\begin{eqnarray*}
\varepsilon(T)&\geq&\sum_{i=1}^{\frac {c-1} 2} \frac{n_i+n_{c-i}+2} 2+\frac {n_c+1} 2+ b+\frac{n_a-1} 2\\
&=& \sum_{i=1}^{c+1} \frac{n_i} 2+ (c-1)+ b \\
&=& \frac{n+1} 2+ c-\frac 3 2\\
&>& \frac{n+1} 2.
\end{eqnarray*}

The proof is completed.
\end{proof}

Recall that the well-known Vizing's theorem says that $\Delta(G)\leq \chi(L(G))\leq \Delta(G)+1$ for any simple graph $G$.

\begin{theorem} If $G$ is a graph of size $m\geq 2$ with no component of order at most two, then $$f_{o}(L(G))\geq \frac{m}{\chi(L(G))}.$$
\end{theorem}
\begin{proof} Let $k=\chi(L(G))$.
We may assume that $G$ is a connected graph of order $n$. Take a spanning tree $T$ of $G$.

\vspace{2mm}\noindent{\bf Observation 1.} If $f_{o}(L(G))\geq \frac n 2$, we are done.
Since $\chi(L(G))\geq \Delta(G)$ and $n\Delta(G)\geq 2m$, we have $$ \frac n 2\geq \frac m {\Delta} \geq \frac m {k}.$$

\vspace{2mm}\noindent{\bf Case 1.} $n\equiv 0\ \text{or}\ 3$ (mod 4)

\vspace{2mm} By Theorem \ref{tree}, $f_o(L(T))\geq \frac n 2$, and thus there is nothing to be shown.

\vspace{2mm}\noindent{\bf Case 2.} $n\equiv 1$ (mod 4)

\vspace{2mm} Let ${M_1, \ldots, M_k}$ be a vertex coloring of $L(G)$ with $k$ colors. Since $M_i$ is a matching of $G$ and $n$ is odd, $|M_i|\leq \frac {n-1} 2$ for each $i\in\{1, \ldots, k\}$. Hence, $\frac {n-1} 2 k=\sum_{i=1}|M_i|=m$, i.e. $\frac {n-1} 2\geq \frac m k$. Combining this with $f_o(L(T))\geq \frac {n-1} 2$ according to Theorem \ref{tree}, we have $f_{o}(L(G))\geq \frac{m}{\chi(L(G))}$.

\vspace{2mm}\noindent{\bf Case 3.} $n\equiv 2$ (mod 4)

\vspace{2mm} By Theorem \ref{tree}, $f_o(L(T))\geq \frac {n-2} 2$. If $\frac{n-2}{2}\geq \frac{m}{k}$, then we are done. Next, assume that $\frac{n-2}{2}<\frac{m}{k}$. Since $n$ is even, it is equivalent to \begin{equation}2m\geq (n-2)k+2.\end{equation}
We further assume that $diam(T)$ is maximum among all spanning trees of $G$. By the choice of $T$, if $T\cong K_{1,n-1}$, then $G\cong K_{1,n-1}$. It is easy to see that $L(G)\cong K_{n-1}$, and thus $f_o(L(G))\geq n-2\geq 1=\frac{n-1}{n-1}=\frac{m} {k}$. So, we consider the case when $diam(T)\geq 3$. Take a longest path $P=u_1u_2\cdots u_p$ of $T$. Clearly, $p=diam(T)+1\geq 4$. Let $N_T(u_2)\setminus \{u_3\}=\{v_1,v_2,\cdots, v_s\}$, where $v_1=u_1$, $d_T(v_i)=1$ for each $i\in\{1, \ldots, s\}$.
Label the two components of $T-u_2u_3$ as $T_{u_2}$ and $T_{u_3}$.
Since $\varepsilon(T)\geq \varepsilon(T_{u_2})+\varepsilon(T_{u_3})$,
$T_{u_2}\cong K_{1,s}$, by Theorem 3.2 $\varepsilon(T_v)\geq \frac{n-(s+1)-2}{2}$ and
$$
\varepsilon(K_{1,s})=\left \{
\begin{array}{ll}
s, & \mbox{if $s\equiv 0$ $(\hspace{-4mm}\mod 2)$}\\
s-1, &\mbox{if $s\equiv 1$ $(\hspace{-4mm}\mod 2)$,}
\end{array}
\right.
$$
we have
\begin{equation} f_{o}(L(G))\geq \varepsilon(T)\geq \left \{
\begin{array}{ll}
\frac {n}{2}, & \mbox{if $s\geq 5$ }\\
\frac {n+1}{2}, & \mbox{if $s=4$  }\\
\frac {n+2}{2}, &\mbox{if $s=2$. }
\end{array}
\right.
\end{equation}

By Observation 1, $f_{o}(L(G))\geq\frac{m}{\chi(L(G))}$ for each of the above cases listed in (11). Thus, it remains to consider $s\in \{1,3\}$.

\vspace{2mm}
A vertex partition $(V_1,V_2)$ of a connected graph $G$ is called a {\it $(3,3)$-partition} if $G[V_i]$ is connected and $|V_i|\equiv 3$ (mod 4) for each $i\in \{1,2\}$. If $G$ has such a partition, then by the induction hypothesis and Theorem \ref{tree}, $$f_o(L(G))\geq \varepsilon(T_1)+\varepsilon(T_2)\geq \frac {|V_1|+1} 2+\frac {|V_2|+1} 2=\frac {n+2} 2>\frac{m}{\Delta(G)}\geq \frac{m}{\chi(L(G))},$$ where $T_i$ is a spanning tree of $G[V_i]$ for each $i\in \{1,2\}$. So, next we complete the proof by showing either $f_o(L(G))\geq \frac n 2$ or the existence of (3,3)-partition of $G$.

\vspace{2mm}\noindent{\bf Subcase 3.1.} $s=3$

\vspace{2mm} Let $V_1=\{v_1,v_2,v_3\}$ and $V_2=V(G)\setminus V_1$.
If $G[V_1]$ is connected, then $(V_1,V_2)$ is a $(3,3)$-partition of $G$, and we are done.
So, assume that $G[V_1]$ is disconnected. Clearly, $V_1$ is an independent set of $G$, otherwise, one can find a spanning tree $T'$ with $diam(T')>diam (T)$. If $e_G(V_1, V(G)\setminus (\{v\}\cup V_1))=0$, then \begin{equation}2m\leq 3+(n-3)\Delta(G).\end{equation} Combining (10) and (12), one has $\Delta(G)\leq 1$, which contradicts $\Delta(G)\geq d_T(v)\geq 4$.
Thus, one of $\{v_1, v_2, v_3\}$, say $v_1$, has a neighbor belonging to $V(G)\setminus (\{v\}\cup V_1)$ in $G$. This results in a $(3,3)$-partition $(V_1',V_2')$ of $G$, where $V_1'=\{v_2, v_3, v\}$.

\vspace{2mm}\noindent{\bf Subcase 3.2.} $s=1$

\vspace{2mm} It means $d_T(u_2)=2$. By symmetry, we may further assume that $d_T(u_p)=1$ and $d_T(u_{p-1})=2$. By the choice of $T$ and $P$, if $p=4$, then $n=p=4$, and if $p=5$, then $T$ is isomorphic to a wounded spider, as shown in Fig. 5.
\begin{center}
\scalebox{0.25}[0.25]{\includegraphics{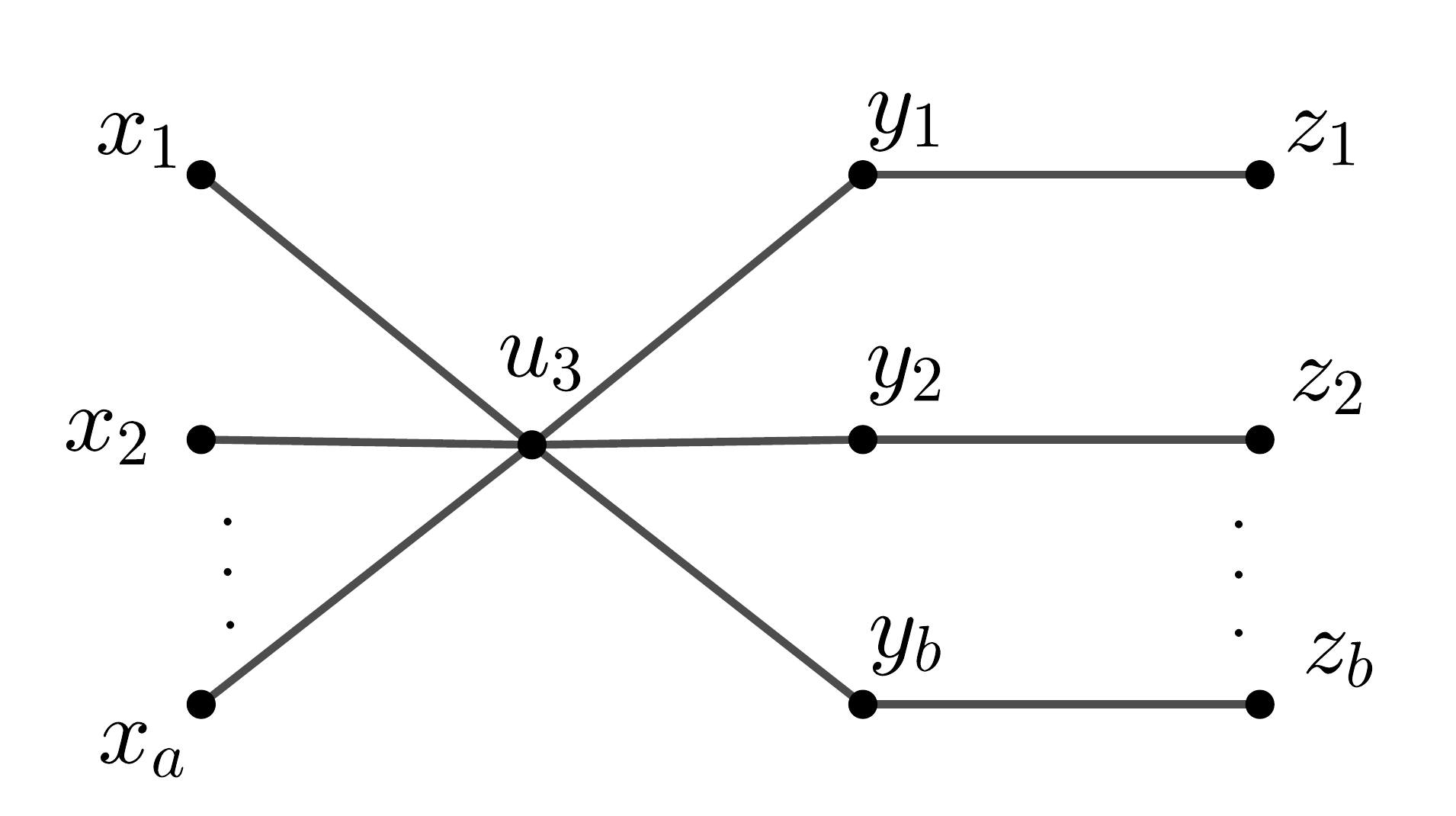}}
\centerline{Fig. 5.~The wounded spider}
\end{center}
For the former case ($p=4$), since $n=4$, it can be verified that $f_0(L(G))=2$ and $2\leq k\leq 3$; moreover, $k=2$ if and only if $G\in\{C_4, P_4\}$, $k=3$ if and only if $\Delta(G)=3$, and thus $f_o(L(G))\geq \frac m k$.

Now let $p=5$.
Since $a+2b+1=n\equiv 2 (\hspace{-2mm}\mod 4)$, $a$ is odd. Thus, if $a+b$ is even, then $\varepsilon(G)\geq a+b\geq \frac{a+2b+1} 2=\frac n 2$. So, let $a+b$ be odd. If $a\geq 3$, then $\varepsilon(G)\geq a+b-1\geq \frac{a+2b+1} 2=\frac n 2$. For $a=1$, if $b=2$, then $\varepsilon(G)=2\geq \frac 5 3=\frac m k$; if $b\geq 4$, $\varepsilon(G)\geq 2b-2\geq \frac{2b+2} 2=\frac n 2$.

Finally, let $p\geq 6$.
If $d_T(u_3)=2$, then $T$ has a (3,3)-partition $(V_1,V_2)$, where $V_1=\{u_1,u_2,u_3\}$ and $V_2=V(G)\setminus V_1$. So, let $d_T(u_3)\geq 3$. In what follows, suppose $G$ has no (3,3)-partition.
Let $v_1, \ldots, v_a$ be the isolated vertices and $w_1r_1, \ldots, w_br_b$ be the isolated edges of $T-u_3$ where $r_i\in N_T(u_3)$ for each $i$, $r_1=u_2$ and $w_1=u_1$.
Let $V'=\{v_i: 1\leq i\leq a\}$, $R'=\{r_i: 1\leq i\leq b\}$ and $W'=\{w_i: 1\leq i\leq b\}$.
By the choice of $T$ and $P$, we have $G[V'\cup W'\cup R']\cong T[V'\cup W'\cup R']$ and $e_G[V'\cup W'\cup R', V(G)\setminus (V'\cup W'\cup R'\cup\{ u_3\})=0$. It follows that for any $i\in \{1, \ldots, a\}$ and $j\in \{1, \ldots, b\}$,
\begin{equation}
d_G(v_i)=1, d_G(r_j)\leq 2, d_G(w_j)\leq 2.
\end{equation}

By exchanging the roles of $u_1, u_2, u_3$ and that of $u_p, u_{p-1}, u_{p-2}$, we make the following assumption as above.
Let $x_1, \ldots, x_c$ be the isolated vertices and $y_1z_1, \ldots, y_dz_d$ be the isolated edges of $T-u_{p-2}$, where $y_i\in N_T(u_{p-2})$ for each $i$, $y_1=u_{p-2}$ and $z_1=u_p$.
Let $X'=\{x_i: 1\leq i\leq c\}$, $Y'=\{y_i: 1\leq i\leq d\}$ and $Z'=\{z_i: 1\leq i\leq d\}$.
By the choice of $T$ and $P$, we have $G[X'\cup Y'\cup Z']\cong T[X'\cup Y'\cup Z']$ and $e_G[X'\cup Y'\cup Z', V(G)\setminus (X'\cup Y'\cup Z'\cup\{ u_{p-2}\})=0$.
It follows that for any $i\in \{1, \ldots, c\}$ and $j\in \{1, \ldots, c\}$,
\begin{equation}
d_G(x_i)=1, d_G(y_j)\leq 2, d_G(z_j)\leq 2.
\end{equation}
One can see the local structure of $T$ around $u_3$ and $u_{p-2}$ in Fig. 6.
\begin{center}
\scalebox{0.32}[0.32]{\includegraphics{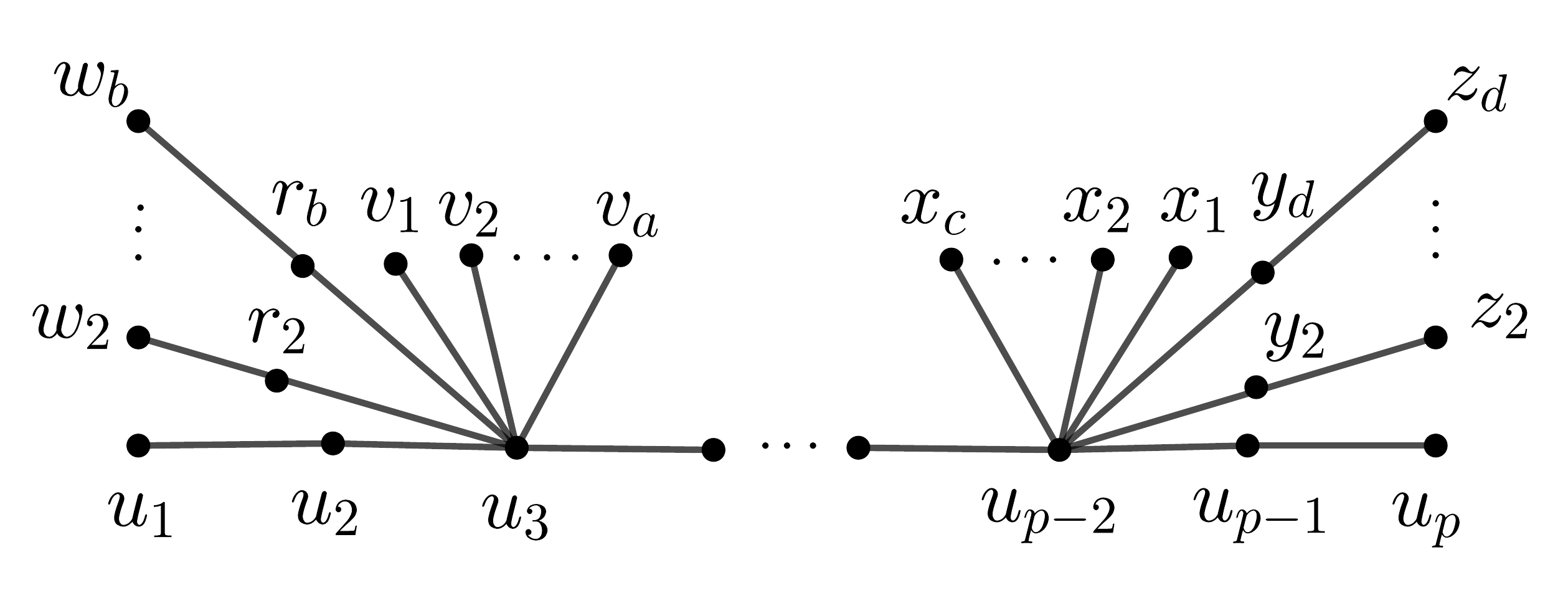}}
\centerline{Fig. 6.~The local structure of $T$ around $u_3$ and $u_{p-2}$}
\end{center}

By (10), (13) and (14), we have $$a+4b+c+4d+(n-a-2b-c-2d))\Delta(G)\geq 2m\geq (n-2)\Delta+2,$$ i.e.
\begin{equation}
a+4b+4d\geq (a+c+2b+2d-2)\Delta(G)+2.
\end{equation}
Since $\Delta(G)\geq d_T(u_3)\geq 3$, by (15), we have $a+b+c+d\leq 2$, which contradicts $a+b=d_T(u_3)-1\geq 2$ and $c+d=d_T(u_{p-2})-1\geq 2$.
This proves $G$ has a (3,3)-partition. By the same reasoning as in the discussion before Subcase 3.1, we conclude that $f_o(L(G))\geq \frac{m} {\chi(L(G))}$.
\end{proof}

Let $p_3(G)$ be the maximum number of vertex disjoint $P_3$ in $G$.
\begin{theorem}[Kelmans and  Mubayi \cite{Kelmans2004}]\label{Kelmans}
If $G$ is a cubic graph of order $n$, then $p_3(G)\geq \frac n 4$.
\end{theorem}

\begin{corollary}
If $G$ is a cubic graph of order $n$, then $f_{o}(L(G))\geq\frac{n}{2}.$
\end{corollary}
\begin{proof}
Let $G$ be a cubic graph of order $n$. By Theorem \ref{Kelmans}, $G$ has at least $\frac n 4$ vertex disjoint $P_3$, then $L(G)$ contains $\frac n 4$ induced $K_2$. Hence, $f_{o}(L(G))\geq\frac{n}{2}$.
\end{proof}

It is an interesting problem whether or not $f_{o}(L(G))\geq\frac{n}{2}$ for any connected regular graph of order at least 3?

\section{\large Disproof of a related conjecture}
In 1997, Berman, Wang, Wargo proposed a conjecture as follows.
\begin{conjecture}[Berman, Wang, Wargo \cite{Berman}] For any connected simple graph $G$ of order $n$, $f_o(G)\geq 2\lfloor\frac{n}{4}\rfloor$.
\end{conjecture}

We disprove it by showing the following theorem.
Since $(\frac n 2 -1)\times \frac n 2<2(\frac{n(n-1)}{8}-1)$ for $n>8$, and $(\frac n 2 -1)\times \frac n 2<2\lfloor\frac{n(n-1)}{8}\rfloor$ for $n=8$,
we have $f_o(G_n)< 2\lfloor\frac{n(n-1)/2} 4\rfloor$.

\begin{theorem}
Let $G_n=L(K_n)$ for any positive integer $n$. If $n$ is even, $f_{o}(G_n)=\frac {n(n-2)} 4$.
\end{theorem}
\begin{proof}

Let $H$ be a maximum odd-even edge induced subgraph of $K_n$ with bipartition $(X,Y)$. Without loss of generality, let $d_H(x)$ be odd for any $x\in X$ and $d_H(y)$ be even for any $y\in Y$.
Since $\sum_{x\in X}d_H(x)=\sum_{y\in Y}d_H(y)\equiv 0\ (\hspace{-2mm}\mod 2)$, $|X|$ must be even. Moreover, since $|X|+|Y|=n$ is even, $|Y|$ is even. It follows that $e(H)=\sum_{x\in X}d_H(x)\leq|X|(|Y|-1)\leq (\frac{|X|+|Y|-1} 2)^2\leq\frac {(n-1)^2} 4=\frac {n^2-2n+1} 4$. Since $e(H)$ is an integer, $e(H)\leq \frac {n(n-2)} 4$. On the other hand, $K_n$ has an odd even edge induced subgraph $K_{\frac n 2, \frac n 2 -1}$ with size $\frac {n(n-2)} 4$. Thus, $f_o(G_n)=\frac {n(n-2)} 4$.
\end{proof}

\section{\large Concluding remarks}

It was shown that Conjecture 2.1 is not true for bipartite graph. It is an interesting problem that whether Conjecture 2.1 holds for graphs with chromatic number at least three. 

Recall that a graph $H$ is called {\it odd-even} if the parity of $d_H(x)$ and $d_H(y)$ is distinct for any edge $xy\in E(H)$. It is clear that an odd-even graph is bipartite. The maximum size of an odd-even edge induced subgraph of a graph $G$ is denoted by $\varepsilon(G)$. Indeed, $\varepsilon(G)=f_o(L(G))$. We have shown that for any tree $T$ of order $n$, $\varepsilon(T)$ is at least $f(n)$, where $f(n)$ is close to $\frac {n-1} 2$ (see Theorem 3.1). It is well known that any graph of size $m$ contains a bipartite graph of size at least $\frac m 2$. Theorem 4.2 indicates that $\varepsilon(K_n)$ is also close to half of the size of $K_n$. It is interesting to give a good lower bound for $\varepsilon(G)$ in terms of the size of $G$.

\vspace{3mm}
\noindent{\bf Data Availability} 

\vspace{2mm}No data was used for the research described in the article.

\vspace{4mm}\noindent{\bf Acknowledgments}

\vspace{2mm} We thank the reviewers whose comments have helped to considerably improve the presentation of the paper.

\end{document}